\documentclass[a4paper,10pt]{article}
\usepackage{graphicx,xcolor,textpos,a4wide}
\usepackage{helvet}

\usepackage{a4wide}

\usepackage{amsfonts}
\usepackage{amssymb}
\usepackage{amsmath}
\usepackage{amsthm}
\usepackage{graphicx}
\usepackage{subcaption}
\usepackage{multicol}
\usepackage{siunitx}
\usepackage{url}
\usepackage[utf8]{inputenc}
\usepackage[nottoc]{tocbibind}
\usepackage{tikz}
\usepackage[symbol]{footmisc}
\usepackage{xcolor}
\usepackage{mleftright}
\usepackage{verbatim}
\usepackage[ruled,vlined]{algorithm2e}
\usepackage[noend]{algpseudocode}
\usepackage{lipsum}
\usepackage{mathrsfs}
\usepackage{tikz-cd}
\usepackage{hyperref}
\hypersetup{
	colorlinks=true,
	linkcolor=blue,
	filecolor=magenta,      
	urlcolor=cyan,
}
\usepackage{tabularx}
\usepackage{multirow}
\usepackage{longtable}
\usepackage{emptypage}
\usepackage{quiver}
\usepackage{dsfont}
\usepackage[capitalise,noabbrev]{cleveref}
\usepackage{cancel}
\usepackage{todonotes}

\newtheorem{theorem}{Theorem}[section]
\newtheorem*{theorem*}{Theorem}
\newtheorem{lemma}[theorem]{Lemma}
\newtheorem{definition}[theorem]{Definition}
\newtheorem*{definition*}{Definition}
\newtheorem{proposition}[theorem]{Proposition}
\newtheorem*{proposition*}{Proposition}
\newtheorem{example}[theorem]{Example}
\newtheorem{corollary}[theorem]{Corollary}

\newtheorem{question}{Question}
\newtheorem{notation}[theorem]{Notation}
\newtheorem{result}{Theorem}

\newtheorem*{result*}{Theorem}
\crefalias{result}{theorem}

\newcommand{\QQ}{\mathbb{Q}}
\newcommand{\RR}{\mathbb{R}}

\newcommand{\ZZ}{\mathbb{Z}}

\newcommand{\NN}{\mathbb{N}}

\newcommand{\id}{\mathds{1}}

\newcommand{\inv}{^{-1}}
\newcommand{\comp}{\circ}

\newcommand{\abs}[1]{\left|#1\right|}
\newcommand{\norm}[1]{\left|\left|#1\right|\right|}
\newcommand{\gnorm}[1]{\left|\left|#1\right|\right|_{\mathrm{g}}}
\newcommand{\set}[1]{\left\{#1\right\}}

\newcommand{\mf}[1]{\mathfrak{#1}}
\newcommand{\mc}[1]{\mathcal{#1}}

\newcommand{\normal}{\vartriangleleft}
\newcommand{\conj}{\mathrm{Conj}}
\newcommand{\cd}{{\mathrm{D}_{\sim}}}

\newcommand{\range}[1]{\{1,2,\cdots,#1\}}

\newcommand{\ind}[1]{\mathrm{Ind}(#1)}

\newcommand{\const}[1]{C_{#1}}

\author{Jonas Der\'e, Lukas Vandeputte}
\title{Uniform Upper Bounds for Conjugacy Separability Growth in Nilpotent Groups}

\date{\today}

\begin{document}
	\maketitle
	\begin{abstract}
    A group is conjugacy separable if every pair of non-conjugate elements remains non-conjugate in some finite quotient. Virtually polycyclic groups, and hence in particular all finitely generated nilpotent groups, are conjugacy separable. The conjugacy separability growth function measures the complexity of distinguishing non-conjugate elements in finite quotients by giving the smallest order of a finite quotient $Q$ that separates them. Recent work showed that this function admits polynomial upper and lower bounds for nilpotent groups, but these estimates are neither explicit nor optimal. We prove improved polynomial upper bounds whose degree is at most linear in the nilpotency class and at most quadratic in the Hirsch length. For $2$-step nilpotent groups, we obtain sharper estimates and show that these estimates are optimal among bounds depending only on the nilpotency class and Hirsch length. This gives a correction of an error in the existing literature. The main tool is a translation of conjugacy separability growth into Lie rings, generalizing previous work in the case of residual finiteness.
\end{abstract}
	\section{Introduction}
	A group $G$ is residually finite if for every non-trivial element $g\in G$ there exists an epimorphism $\varphi:G\to Q$ onto a finite group such that $\varphi(g)\neq 1$. Mal'cev showed in \cite{Malcev1983} that if a finitely presented group is residually finite, then it has solvable word problem. Similarly, a group $G$ is conjugacy separable if for every pair of non-conjugate elements $g,h\in G$, there exists an epimorphism to a finite group $\varphi : G\rightarrow Q$ such that $\varphi(g)$ and $\varphi(h)$ are non-conjugate in $Q$. Again, if a finitely presented group is conjugacy separable, it has solvable conjugacy problem. 
	
Classically, a lot of effort was put into determining which groups satisfy these separability properties. For example, free groups, surface groups, virtually polycyclic groups, fundamental groups of closed orientable $3$-manifolds and free-by-cyclic groups are known to enjoy both properties, see \cite{blackburn1965conjugacy,formanek1976conjugate,hamilton2013separability,hempel1987residual,remeslennikov1971finite,stebe1972conjugacy,dahmani2025free}. The quantitative study of these properties was initiated by Bou-Rabee in \cite{bou2010quantifying}, where he introduced the residual finiteness function $\mathrm{RF}_G:\NN\to\NN$. Given $n\in\NN$, the value $\mathrm{RF}_G(n)$ is the minimal integer such that for every non-trivial $g\in G$ with $\norm{g}_S\le n$, there exists an epimorphism $\varphi:G\to Q$ onto a finite group with $|Q|\le \mathrm{RF}_G(n)$ and $\varphi(g)\neq 1$. Although $\mathrm{RF}_G$ depends on the generating set, its asymptotic growth does not. Similarly, for $n\in\NN$, let $\conj_G(n)$ denote the smallest integer such that for every $a,b\in G$ that are non-conjugate and of word norm at most $n$, there exists a quotient $\varphi: G\rightarrow Q$ with $\abs{Q}\leq \conj_G(n)$ in which $\varphi(a)$ and $\varphi(b)$ remain non-conjugate. Thus $\conj_G(n)$ measures the size of the finite quotient required to certify non-conjugacy of elements of word length at most $n$. As for residual finiteness growth, the asymptotic growth type of $\conj_G$ is independent of the chosen finite generating set.
	
	In general, the asymptotic behavior of the functions $\mathrm{RF}_G$ and $\conj_G$ is poorly understood. While upper and lower bounds have been established for many groups, these bounds are often far from sharp. The function $\conj_G$ is only known in rare cases such as for the abelian groups \cite{bou2010quantifying}, the lamplighter groups \cite{lawton2017decision}, and more recently the virtually abelian groups \cite{dere2024virtab} as well as a class of groups constructed to exhibit prescribed conjugacy separability growth \cite{vandeputte2026arbitrary}. More background can be found in the survey \cite{dfp2024}. In this article, we study the conjugacy separability growth of torsion-free nilpotent groups.
	
	In the nilpotent case, little is known. In \cite{pengitore2015effective}, Pengitore showed that when $G$ is a generalised Heisenberg group, then $\conj_G(n)\simeq n^{h(G)}$ where $h(G)$ is the Hirsch length of $G$. More generally, it has been shown in \cite{dere2019effective} that $\conj_G$ is bounded below by $n^3$ if $G$ is not virtually abelian and bounded above by some polynomial. The proof in \cite{dere2019effective} relies on induction along a central series, leading to bounds whose degree grows rapidly with the nilpotency class and Hirsch length. These estimates are therefore far from optimal in general. Our main result provides a sharper polynomial upper bound whose degree depends explicitly on natural invariants of the group.
	\begin{result}\label{prop:AlgUpperBound}
		Let $G$ be a torsion-free finitely-generated $k$-step nilpotent group, then
		$$\conj_G(n)\preceq n^m$$ where \[ m=(k-1)\bigl(h(G)-h(Z(G))+1\bigr)h\bigl([G,G]\bigr). \]
	\end{result}
\noindent	Here $h$ denotes the Hirsch length, a classical invariant for nilpotent groups, generalizing the rank of abelian groups. In particular, the exponent is linear in the nilpotency class and quadratic in the Hirsch length, a significant improvement over the previously known bounds of \cite{dere2019effective}.

	In the $2$-step nilpotent case, the authors of \cite{dere2019effective} further claim a precise expression for the conjugacy separability growth. Their claim would imply that $$
	\conj_G(n)\preceq n^{h(G)}.
	$$
	However, the proof relies on \cite[Theorem 1.1]{pengitore2015effective}, which was later disproved. We emphasize that the remaining results of \cite{dere2019effective} do not rely on this theorem. In this situation, we obtain the following estimate, sharper than \ref{prop:AlgUpperBound}:	
	\begin{result}\label{prop:AlgUpperBound2}
		Let $G$ be a torsion-free finitely-generated $2$-step nilpotent group, then $$
		\conj_G(n)\preceq n^m
		$$
		where $$m=\max \left\{ h(G),h\left([G,G]\right) ~ \lfloor h(G)-h(Z(G))\rfloor_{\mathrm{odd}}\right\}$$
	\end{result}
	The term $h(G)$ is only needed when $h([G,G])=1$, for groups with larger derived subgroup, the second term dominates. Moreover, this estimate is sharp for a family of groups whose abelianization and derived subgroup have arbitrarily large rank. In particular, we obtain infinitely many examples for which $n^{h(G)} \prec  \conj_G(n),$ showing that the proposed expression $\conj_G(n)\preceq n^{h(G)}$ fails in general.
    
Our main tool, which is of independent interest, is a translation of conjugacy separability growth to properties of Lie rings, allowing to apply methods from Lie theory to compute conjugacy separability. This extends the ideas from \cite{dmv2026}, where these were used to give a precise computation of $\mathrm{RF}_G$ for virtually nilpotent groups. In \Cref{sec:AlgUpperBoundGeneral} a partial translation is given for arbitrary nilpotency classes. Although this does not completely translate the problem, it already gives the upper bound in \Cref{prop:AlgUpperBound}. Next, in \cref{sec:groupsandLierings}, we focus on the case of $2$-step nilpotent groups, where we can completely translate our problem and reformulate it in the language of Lie rings:
	\begin{result}
		\label{prop:conjequal}
		Let $G$ be a non-abelian torsion-free 2-step nilpotent group with corresponding Lie ring $\mf g$. Let $S^\prime$ and $S$ be generating sets for $\mf g$ and $G$ respectively, then $$\conj_{\mf g,S^\prime}(n)\simeq\conj_{G,S}(n).$$
	\end{result} 
	Here $\conj_{\mf g,S^\prime}$ is the appropriate analogue of $\conj_{G,S}$, which we introduce in Definition \ref{sepinLiering}. This restatement allows us to demonstrate \cref{prop:AlgUpperBound2}. The translation to the setting of Lie rings will then also allow us to demonstrate the sharpness of \Cref{prop:AlgUpperBound2} in \cref{sec:AlgBoundScharp}.

	\section{Preliminaries}
	\subsection{Effective separability}
	\numberwithin{theorem}{subsection}
	In a group $G$, two elements $g_1$ and $g_2$ are called \textbf{conjugate}, denoted $g_1\sim g_2$, if there exists some $h\in G$ such that $hg_1h\inv=g_2$. Analogously we denote $g_1\not\sim g_2$ if $g_1$ and $g_2$ are non-conjugate. A group $G$ is then called \textbf{conjugacy separable} if for all pairs $g_1,g_2\in G$ of non-conjugate elements, there exists a finite quotient $\varphi: G \to Q$ such that $\varphi(g_1)\not\sim\varphi(g_2)$.
	The aim is to quantify how the size of the finite quotient $\varphi(G)$ depends on $g_1$ and $g_2$.
	A basic tool is the word norm.
	
	\begin{definition}
	Let $S$ be a finite generating set of a group $G$. The \textbf{word norm} $ \|\cdot\|_S:G\to\NN $ is defined by $\|e_G\|_S=0$ and $\|g\|_S = \min\Bigl\{ n\in\NN \;\Big|\; g=s_1\cdots s_n,\; s_i\in S\cup S^{-1} \Bigr\}$ for $g \neq e_G$.
	\end{definition}
	A change of generating set only changes the word norm up to a constant, see \cite{MR3729310}.
	\begin{lemma}
		\label{prop:normAreEquiv}
		Let $G$ be a finitely generated group and let $S$ and $T$ be two generating subsets of $G$. There exists a constant $C > 0$ such that for every group element $g$ the following holds: $$\norm{g}_T\leq C\norm{g}_S $$
	\end{lemma}

	Given the word norm, we can introduce the following two functions, both measuring the size of the finite quotients:
	\begin{definition}
		Let $G$ be a conjugacy separable group, then the \textbf{conjugacy depth function} $\cd_G: G\times G\rightarrow \NN$ is defined as $\cd_G(g_1,g_2) = 1$ if $g_1 \sim g_2$, otherwise $$\cd_G(g_1,g_2) = \min\set{[G:N]\mid N\normal G, g_1N\not\sim g_2N \text{ in } G/N}.$$
	\end{definition}
	\begin{definition}
		Let $G$ be a finitely generated conjugacy separable group with a finite generating set $S$, then we define the \textbf{conjugacy separability growth function} $\conj_{G,S}$ as:
		$$
		\conj_{G,S}:\NN\rightarrow \NN:n\mapsto \max\set{\cd_G(g_1,g_2)\mid g_1,g_2\in G,\norm{g_1}_S\leq n,\norm{g_2}_S\leq n}.
		$$
	\end{definition}
	Since $G$ is conjugacy separable, it is guaranteed that the above functions only take finite values.
	If we change the generating set, then the word norm and by extension the conjugacy separability growth function might change. To mitigate this, we introduce an equivalence class on such functions.
	
	\begin{definition}
		Let $f,g:\NN_0\rightarrow\RR^+_0$ be non-decreasing functions. Then we say $f\preceq g$ if there exists some constant $C \in \mathbb{N}$ such that $f(n)\leq Cg(Cn)$ for all $n\in \NN_0$. Furthermore if $f\preceq g$ and $g\preceq f$, then we say $f\simeq g$
	\end{definition}
		The above definition deals only with positive functions. If we want to compare to a function that also takes negative values (for instance $\log$), then we modify our equivalence relation as follows. If $f,g:\NN_0\rightarrow\RR$ are non-decreasing functions, then we denote $f(n)\preceq g(n)$ when $\max(1,f(n))\preceq \max(1,g(n)).$ 	When $f$ and $g$ only take strictly positive values, then the two definitions are equivalent.
        
        The ambiguity in the choice of generating set, can now be completely suppressed.
	\begin{lemma}\label{prop:generatorinvariance}
		Let $G$ be a finitely generated conjugacy separable group with finite generating sets $S$ and $T$, then $\conj_{G,S}\simeq\conj_{G,T}$.
	\end{lemma}
	This lemma is an immediate consequence of \cref{prop:normAreEquiv}. To compare two functions with $\preceq$, it suffices to understand the function in sufficiently many points. 
	
	\begin{lemma}
		\label{prop:comparegeometric}
		Let $f,g:\NN_0\rightarrow\RR^+_0$ be non-decreasing functions. Let $C>1$ and let $X$ be a set of natural numbers such that for every natural number $m$ it holds that $X \cap [m;Cm] \neq \emptyset$. 
	If $f(x)\leq g(x)$ for all $x\in X$ then $f\preceq g$
	\end{lemma}
	\begin{proof}
		Let $m$ be arbitrary and take $x\in X \cap [m,Cm]$. Then we have the following inequalities:
		\begin{align*}
		f(m) \leq f(x) \leq g(x) \leq g(Cm).
        \end{align*}
		Since this holds for all values of $m$, we thus have $f\preceq g$	
	\end{proof}
	The above lemma states that to check $f\preceq g$, we only need to check this on a sequence that grows more slowly than a geometric sequence. One example is the sequence of primes on an arithmetic progression.
	
		\begin{lemma}\label{prop:psubgeom}
		Let $a$ and $b$ be coprime positive integers. The set $X = \{p \in  a \NN + b \mid p \text{ is prime}\}$ satisfies the conditions of Lemma \ref{prop:comparegeometric}.
	\end{lemma}
	\begin{proof}
		This is a consequence of the prime number theorem for arithmetic progressions as stated and proven in \cite{delavallepoussain1899Recherces}.
	\end{proof}
	We conclude this section by recalling the conjugacy separability growth of the abelian groups, which will be important in certain arguments.
	\begin{proposition}[\cite{bou2010quantifying}]
		Let $G$ be an infinite finitely generated abelian group, then $\conj_G(n)\simeq \log(n)$
	\end{proposition}
	This is a consequence of the following reformulation of the prime number theorem that will be used repeatedly to construct finite quotients of controlled index.
	\begin{lemma}
		\label{prop:pnondevider}
		Let $a$ and $b$ be coprime integers. There exists a constant $C_a>0$, depending only on $a$, such that for every positive integer $n$, there is a prime number $p$ satisfying  $p \equiv b \mod a$, $p\nmid n,$ and $p \leq C_a\log n + C_a.$
	\end{lemma}

	\subsection{Nilpotent groups}
	Throughout the paper, we use the convention that the commutator is given by $[a,b]=aba\inv b\inv$. Let $G$ be a group then the \textbf{lower central series} is the inductively defined sequence of subgroups $\gamma_iG\normal G$ with $\gamma_1G=G$ and $\gamma_{i+1}G=[\gamma_iG,G]$. A group is called \textbf{$k$-step nilpotent} if $\gamma_{k+1}G=1$ but $\gamma_k G\neq 1$.
	More generally, a \textbf{central series} of a nilpotent group is a descending sequence of subgroups $G_i$ such that $G_1=G$, such that $G_{i+1}\supset [G_i,G]$ and such that $G_k=1$ for some $k \in \NN$.
	
	An important invariant when working with nilpotent groups, and more generally polycyclic groups, is the Hirsch length.
	\begin{definition}
		Let $G$ be a $k$-step nilpotent group with lower central series $\gamma_iG$. Then the \textbf{Hirsch length} $h(G)$ is given by the sum of the torsion-free ranks of $\frac{\gamma_iG}{\gamma_{i+1}G}$
	\end{definition}
	\begin{definition}
		Let $G$ be a group and let $H$ be a subgroup. The \textbf{isolator} $\sqrt[G] H$ is the set$$
		\sqrt[G] H=\set{x\in G\mid\exists n\in \NN:  x^n\in H}.
		$$
	\end{definition}
	\noindent In general, this set is not a subgroup, however when $G$ is nilpotent it always is.
	\begin{lemma}
		Let $G$ be a nilpotent group and let $H$ be a normal subgroup, then $\sqrt[G] H$ is a normal subgroup of $G$.
	\end{lemma}
	Let $G$ be a finitely generated nilpotent group with normal subgroup $H$, then $H$ is always finite index in its isolator$\sqrt[G]{H}$ we denote this index by $[\sqrt[G]{H}:H]=\det_G(H)$. We define the adaptive lower central series in terms of the lower central series as $$
	G_i=\sqrt[G]{\gamma_iG}.
	$$
	One shows that when $G$ is torsion-free, this is indeed a central series.
	
	We fix some notation for the Hirsch length of some important subgroups of $G$.
	\begin{notation}
		Let $G$ be a torsion-free finitely generated nilpotent group. As usual we denote by $h(G)$ the Hirsch length of $G$, we denote by $h_2(G)$ for $h(G_2)$, and $h_z(G)$ for $h(Z(G))$, the Hirsch length of the center of $G$.
	\end{notation}

	To end this section we consider what happens in the special case of $2$-step nilpotent groups.	
	A first special property is the following.
	\begin{lemma}
		\label{prop:comutIsMorphism}
		Let $G$ be $2$-step nilpotent and let $g\in G$, then $[.,g]$ and $[g,.]$ are both homomorphisms from $G$ to $[G,G]$.
	\end{lemma}
	In case $G$ is also torsion-free, then the above morphism goes to a torsion-free abelian group and thus $G_2$ must lie in the kernel. It follows that $[.,.]$ induces a bilinear map $\frac{G}{G_2}\times\frac{G}{G_2} \to G_2$.
	
	One obvious consequence of this lemma is that the image of the commutator is a subgroup.
	\begin{corollary}
		Let $G$ be a $2$-step nilpotent group and let $a\in G$ be an arbitrary element. Then $[a,G]$, the set of commutators $\set{[a,g]\mid g\in G}$, is a central subgroup of $G$.
	\end{corollary}

For $2$-step nilpotent groups, conjugacy admits a particularly simple description in terms of commutators.	
\begin{lemma}
		\label{prop:redefineconj}
		Let $G$ be a $2$-step nilpotent group and $a,a'\in G$. Then $a$ is conjugate to $a'$ if and only if $a'a\inv\in [a,G]$.
	\end{lemma}

	\subsection{Lie rings}
A central idea of this paper is to translate questions on
conjugacy separability growth of nilpotent groups into
questions about Lie rings.	\begin{definition}
		A Lie ring over a ring $R$ is an $R$-module $\mf g$, together with a bilinear map:$$
		[\_,\_]:\mf g\times \mf g\rightarrow \mf g
		$$
		such that $[x,x] = 0$ for all $x \in \mf g$ and $[x,[y,z]]+[y,[z,x]]+[z,[x,y]]=0$ for all $x,y,z\in\mf g$.
	\end{definition}
	Although this can be defined over general rings, in this paper it will only be used over the integers $\ZZ$ or $\ZZ/m\ZZ$.
	As can be seen from the definition, when the ring $R$ is a field, then we obtain a Lie algebra. In analogy to a Lie subalgebra, a \textbf{Lie subring} $\mf h$ of a Lie ring $\mf g$ is a submodule such that if $g_1,g_2\in\mf h$, then $[g_1,g_2]\in\mf h$.
	We call such a Lie subring an \textbf{ideal} if for $g_1\in\mf g$ and for $g_2\in\mf h$, then $[g_1,g_2]\in\mf h$. In this case we denote $\mf h\normal\mf g$.
	A \textbf{morphism of Lie rings}, is a morphism of modules that preserves the commutator.
	
	Similar to groups, we can take the isolator of a Lie subring. Let $\mf g$ be a Lie ring over an integral domain $R$ and let $\mf h$ be a Lie subring, then we define the \textbf{isolator} as $$
	\sqrt[\mf g]{\mf h}=\{h\in \mf g\mid \exists r\in R\backslash \{0\}: rh\in\mf h\}
	$$
	Note that this is just the radical of modules. It is immediate from bilinearity that $\sqrt[\mf g]{\mf h}$ is a submodule of $\mf g$. Furthermore, if $\mf h$ is an ideal, then one checks that $\sqrt[\mf g]{\mf h}$ is also an ideal of $\mf g$.
	We call a Lie subring $\mf h$ \textbf{isolated} if $\mf h=\sqrt{\mf h}$.
	
	When we call a set $S$ a \textbf{generating set} or a \textbf{basis} of $\mf g$, then this is considered for $\mf g$ as a $\mathbb{Z}$-module. When $R=\ZZ$, the following notation helps us in splitting a Lie ring in distinct modules.
	\begin{definition}
		Let $M$ be a finitely generated $\ZZ$ module, and $N$ an isolated submodule, then we say a submodule $N^\perp \subset M$ is \textbf{a complement} of $N$ if $N\oplus N^\perp=M$. If $N$ is not isolated then we call a module $N^\perp$ \textbf{a complement} of $N$ if it is a complement of $\sqrt{N}$.
	\end{definition}
	Complements do not depend on the bracket and are not unique.

	\begin{definition}
		Let $\mf g$ be a torsion-free finitely generated nilpotent Lie ring, then we denote $\gamma_i\mf g$ the \textbf{lower central series}, that is $\gamma_1\mf g=\mf g$, and $\gamma_i\mf g$ is defined inductively as $\gamma_{i+1}\mf g=[\gamma_i\mf g,\mf g]$. We then define the \textbf{adaptive lower central series}  by $\mf g_i=\sqrt[\mf g]{\gamma_i\mf g}$. A Lie ring $\mf g$ is said to be $k$-step \textbf{nilpotent} if $\gamma_{k+1} \mf g = 0$ but $\gamma_k \mf g \neq 0$.
	\end{definition}
\noindent The \textbf{center} $Z(\mf g)$ consists of the elements $g\in\mf g$ for which $[g,h]=0$ for all $h \in \mf g$. If $\mf g$ is nilpotent, then $Z(\mf g)$ is always non-trivial.
	
	The usual norms on Lie rings are not very compatible to the norms in the group theory setting, therefore we consider the Guivarch norm as introduced in \cite{guivarch1973croissance}.
	Let $\mf g$ be a nilpotent Lie ring of class $k$ over $\QQ$. Since $\mf g$ is finite-dimensional over $\QQ$, we may choose complements $\nu_i$ of $\mf g_{i+1}$ in $\mf g_i$. This yields a vector space decomposition $ \mf g=\displaystyle \bigoplus_{i=1}^{k}\nu_i $ satisfying $ \mf g_j=\displaystyle \bigoplus_{i=j}^{k}\nu_i$ for all $j$. Equip each $\nu_i$ with a fixed supremum norm $\|\cdot\|_i$ arising from a chosen basis. The \textbf{Guivarch norm} $\gnorm{\_}$ is defined as $$
	\gnorm{x}=\sup_{1 \leq i \leq k} \{\norm{x_i}_i^{\frac{1}{i}}\}
	$$
	where $x_i\in\nu_i$ is such that $x=\sum_{i=1}^k x_i$. This norm depends on the decomposition of $\mf g$ and on the chosen basis for the summands $\nu_i$, but this dependence is only up to a constant as in \cref{prop:normAreEquiv}.
	This norm decreases under morphisms, that is given a morphism $\varphi:\mf g\rightarrow \mf h$, there exists some constant $C>0$ such that $\gnorm{\varphi(x)}\leq C\gnorm{x}$ for all $x\in \mf g$.
	If $\mf g$ is a finitely generated torsion-free nilpotent Lie ring of class $k$, then $\mf g$ embeds into $\QQ\otimes \mf g$. Through this embedding, we can pull back the Guivarch norm to $\mf g$.
		
	\subsection{Separability in modules}
	
The Lie ring arguments developed later ultimately reduce to separating an element from a submodule in a finitely generated $\ZZ$-module, without taking the Lie bracket into account. We therefore collect several elementary estimates that will be used later. Let $M$ be a finitely generated free $\ZZ$-module with basis $S$. For a submodule $V \subset M$, the norm $\Vert V \Vert_S$ is defined as \[ \|V\|_S = \min\left\{ \max_{x\in\mathcal V}\|x\|_S \;\middle|\; \mathcal V \text{ is a finite generating set of } V \right\}. \]
	\begin{lemma}[\cite{dere2019effective} Proposition 4.2]\label{prop:SubspaceNormToIndex}
		Let $M$ be a finitely generated $\ZZ$-module of rank $d$ with basis $S$ and $ V\subset M$ be a submodule with $\Vert V \Vert_S \leq n$. The determinant of $V$ is bounded above by$$
		\det(V) \leq Cn^d
		$$ for some uniform constant $C$.
	\end{lemma}
		\begin{definition}
		Let $M$ be a finitely generated $ \ZZ$-module, and $x\in M$ an element that is not torsion, then the \textbf{index} $\mathrm{Ind}(x)$ is the largest integer $l$ such that there exists some $x_0$ in $M$ satisfying $lx_0=x$.
	\end{definition}
	\begin{corollary}\label{prop:indexInQuotient}
		Let $d$ be a positive integer and $V \subset \ZZ^d$ be a submodule with $\Vert V \Vert_S \leq n$, where $S$ is the standard basis. There exists a constant $C$ such that for all $c\in \ZZ^d\setminus\sqrt{V}$ of  norm at most $n$, the index of $c$ in $\ZZ^d/V$ is bounded above by $C n^d$.
	\end{corollary}
	\begin{proof}
		Write $U = V + \langle c \rangle$. The index of $c$ in $\frac{\ZZ^d}{V}$ is bounded by the order of the torsion subgroup of $\frac{\ZZ^d}{U}$, and thus by the index of $U$ in its isolator. Lemma \ref{prop:SubspaceNormToIndex} implies that $[\sqrt{U}:U] = \det(U) \leq Cn^d$ for some constant $C$ depending only on $d$. 
	\end{proof}
	
	This leads to the following separability result.
	\begin{lemma}\label{prop:AbelianSubgroupSeparability}
    Let $M$ be a finitely generated free $\ZZ$-module of rank $d$ with basis $S$. Take $m \geq n >0$ and $V \subset M$ with $\|V\|_S\le n$ and $c\in M\setminus V$ with $\|c\|_S\le m$. Then there exists a constant $C$, depending only on $d$, and a finite quotient $\varphi:M\to Q$ such that $\varphi(c)\notin\varphi(V)$ and $\varphi(V)=0$. Moreover: \begin{itemize} \item if $c\in\sqrt{V}$, then $|Q|\le \det(V) \leq  Cn^d$, and \item if $c\notin\sqrt{V}$, then $|Q|\le C\log(m) + C$. \end{itemize}
	\end{lemma}
	\begin{proof}
		Let $V^\perp$ be a complement of $V$ in $M$.
		If $c \in \sqrt{V}$, then $c\notin V+V^\perp$. Furthermore, the index of $V+V^\perp$ in $M$ equals $\det(V)$. From \cref{prop:SubspaceNormToIndex} it follows that $\det(V)\leq C_1 n^d$ for some constant $C_1$. The result follows by taking $\varphi$ to be the quotient map with kernel $V+V^\perp$.
		
		For the second case assume that $c\notin \sqrt{V}$. Consider then the projection $\varphi_0:M\rightarrow V^\perp$, with as kernel $\sqrt{V}$. By \cref{prop:indexInQuotient} it follows that $\varphi_0(c)$ has index no larger than $C_2 m^{d}$. By \cref{prop:pnondevider}, it follows that there exists some prime number $p$ bounded by $$C_3\log(C_2 m^d) + C_3 = C_3 d \log(m) + C_3$$ for some constant $C_3$ such that $p$ does not divide the index of $\varphi_0(c)$ or thus such that $\varphi_0(c)\notin pV^\perp$. The result follows by taking $\varphi$ the quotient map with kernel $\sqrt{V}+pV^\perp$, composed with a projection on one of the cyclic components, and rewriting the constants.
	\end{proof}

		\section{General upper bound}\label{sec:AlgUpperBoundGeneral}
		\numberwithin{theorem}{subsection}
        In this section, we present the proof of Theorem \ref{prop:AlgUpperBound}.
	The main tool is the close correspondence between Lie rings and nilpotent groups as given by the Baker-Campbell-Hausdorff formula. By introducing strong ideals, conjugacy is related to the Lie bracket, giving the uniform upper bound.
	\subsection{Relation to Lie rings}
	We first introduce the following notation for nested commutators.
	\begin{definition}
Let $\mc I$ denote the set of all finite tuples $v=(v_1,\dots,v_n)$ of strictly positive integers. For $v\in\mc I$, we define its length by \[ l(v)=\sum_{i=1}^{n} v_i. \] For $k\in\NN$, let \[ \mc I_k=\{v\in\mc I\mid l(v)\le k\}. \]
	\end{definition}
	\begin{definition}
		Let $\mf g$ be a Lie ring and let $v = (v_1,\ldots,v_n)\in \mc I$. For every $a,b\in\mf g$, we define $[a,b]_v$ as$$
		[a,b]_v=[[\cdots[a,\underbrace{b],b],\cdots],b]}_{v_1},\underbrace{a],a],\cdots],a]}_{v_2},\underbrace{b],\cdots],b]}_{v_3},\cdots]
		$$
	\end{definition}
	\noindent A similar convention will be used for commutators whenever needed.
    
    As the backbone of our correspondence, we use the Baker-Campbell-Hausdorf-formula.
	The following is an amalgamation of facts that can be found in \cite[Chapter 6]{segal1983polycyclic}
	\begin{theorem}\label{prop:bch}
		There exist rational numbers $q_v$,$r_v\in\QQ$ such that for any torsion-free finitely generated nilpotent group $G$ of nilpotency class $k$, there exists a torsion-free finitely generated nilpotent Lie ring $\mf g$ such that for any $v\in\mc I_{k}$, both $q_v[\_,\_]$ and $r_v[\_,\_]$ are internal in $\mf g$, and such that there exists an injective map $\log:G\hookrightarrow\mf g$ satisfying the following:\begin{itemize}
			\item There exists a constant $m$ such that $m\mf g\subset\log(G)$,\\
			\item $\log(ab)=\log(a)*\log(b)$ where $x*y$ is given by$$
			x*y=x+y+\sum_{v \in \mc I_k } q_v[x,y]_v,
			$$\\
			\item $\log([a,b])=\{\log(a),\log(b)\}$ where $\{x,y\}$ is given by $$
			\{x,y\}=[x,y]+\sum_{(1) \neq v \in \mc I_k} r_v[x,y]_v
			$$
		\end{itemize}
	\end{theorem}
	One can furthermore show that $(\mf g,*)$ is a group and, as the map $\log$ is injective, we may see $G$ as a subgroup of $(\mf g,*)$. To avoid confusion, henceforth we denote the group commutator, both in $G$ and $(\mf g,*)$ with $\{\_,\_\}$.
	
	Just like with iterated Lie brackets, we might take iterated commutators. 
	\begin{lemma}\label{prop:bchRepeated}
		There exist constants $\tilde{r}_{v,\tilde v}\in \QQ$ for all $v\in \mc I$ such that for all $x,y\in \mf g$, $$
		\{x,y\}_v=[x,y]_v+\sum_{\tilde v \in \mc I_{k} \setminus  \mc I_{l(v)}} \tilde{r}_{v,\tilde v}[x,y]_{\tilde v}.
		$$
	
	\end{lemma}
	\begin{proof}
		This follows from repeatedly substituting $\{\_,\_\}$ into itself.
	\end{proof}
	\begin{corollary}\label{prop:bchGroupCommutator}
		There exist constants $s_v\in \QQ$ for all $v\in \mc I$ such that for all $x,y\in \mf g$:$$
		\{x,y\}=[x,y]+\sum_{(1) \neq v \in \mc I_k} s_v\{x,y\}_v
		$$
	\end{corollary}
	\begin{proof}
		This follows from substituting the expressions from \cref{prop:bchRepeated} into \cref{prop:bch}.
	\end{proof}
	
	In what follows we fix a torsion-free finitely generated nilpotent group $G$, and corresponding Lie ring $\mf g$. 
	Strong ideals help to translate notions of Lie rings into nilpotent groups.
	\begin{definition}
		Take notations as in \Cref{prop:bch}, then we call an ideal $I\normal \mf g$ a \textbf{strong ideal} if $I\subset G$ and if for any $v\in\mc I_{k}$, for any $x\in I$ and $y\in \mf g$, we have that $q_v[x,y],r_v[x,y]\in I$.
	\end{definition}
Note that if $I$ is a strong ideal, then automatically also the weaker condition $q_v[x,y]_v\in I$ and $r_v[x,y]_v\in I$ is satisfied for all $x \in I$ and $y \in \mf g$, which is in fact the only condition needed outside of \cref{prop:ConjTranslation}. However, it is easier to show that ideals satisfy the condition on the Lie bracket than on nested Lie brackets in Lemmata \ref{prop:IdealContainsStrong} and \ref{prop:SubgroupContainsStrong}.
    
	\begin{example}\label{eg:strongIdeal}
		Take notations as in \Cref{prop:bch}, then $m\mf g$ is a strong ideal.
		Indeed $m\mf g\subset G$ is by choice of $m$ and $m\mf g$ is clearly an ideal of $\mf g$. Furthermore, for $x, y\in\mf g$ we have $q_v[mx,y]=m(q_v[x,y])\in m\mf g$ and similarly, for the term with $r_v$
	\end{example}
	The advantage of working with strong ideals, is that they behave nicely, both as ideals in $(\mf g,+)$, and as subgroups in $(G,*)$. This is made concrete in the following lemmas.
	\begin{lemma}\label{prop:StrongIdealIsNormal}
		Take notations as in \Cref{prop:bch} and let $I$ be a strong ideal, then $I$ is a normal subgroup of $G$.
	\end{lemma}
	\begin{proof}
	This follows immediately from \cref{prop:bch}, as $\{x,y\} \in I$ for all $x \in \mf g, y \in I$.
	\end{proof}
	\begin{lemma}\label{prop:CosetsEquivalent}
		Take notations as in \Cref{prop:bch} and let $I$ be a strong ideal, for any $c\in G$ the subsets $c+I = c*I$ are equal.
	\end{lemma}
	\begin{proof}
		First let $y\in I$ be arbitrary, then $c*y$ is given by
		$$
		c*y=c+y+\sum_{v \in \mc I_k} q_v[c,y]_v.
		$$

		By definition of a strong ideal clearly $c*y\in c+I$.
		
		For the converse direction we proceed by induction.
		Denote $I_i=I\cap\mf g_{i}$. First notice that if $y\in I_{k}$, then $[c,y]_v$ vanishes for all $v\in\mc I$ and thus $c+y=c*y$. Assume now that for any $c\in G$ and for any $y\in I_i$ there exists some $y'\in I_i$ (depending on $c$) such that $c+y=c*y'$. Let $y\in I_{i-1}$ be arbitrary.
		Then we have by \cref{prop:bch} that $$
		c*y=c+y+\sum_{v \in \mc I_k} q_v[c,y]_v
		$$
		or thus rearranging terms
		$$c+y=c*y-\sum_{v \in \mc I_k} q_v[c,y]_v$$
		As $\displaystyle \sum_{v \in \mc I_k} q_v[c,y]_v\in I_{i}$, we have that, using the induction hypothesis on $c*y$ and $- \displaystyle \sum_{v \in \mc I_k} q_v[c,y]_v$, there exists some $y'\in I_i$ such that $c+y=(c*y)*y'=c*(y*y')$
		By induction, it follows that for any $y\in I$, $c+y\in c*I$ as had to be shown.
	\end{proof}
	
	\begin{corollary}\label{prop:IndexCorrespondence}
		Take notations as in \Cref{prop:bch}, then there exists some constant $M$ such that for $I$ a strong ideal $$
		[\mf g:(I,+)]\geq[G:(I,*)]\geq 1/{M}[\mf g:(I,+)]
		$$
	\end{corollary}
	\begin{proof}
		By \cref{prop:CosetsEquivalent}, the map $f: G/I\rightarrow \mf g/I$ mapping $g*I$ to $g+I$ is injective and well-defined.
		By \cref{prop:bch} there exists some constant $m$ such that $m\mf g\subset G$. Again by \cref{prop:CosetsEquivalent}, the map $g:(m\mf g+I)/I\longrightarrow G/I$ is also injective and well-defined.
		The result follows for $M = m^{h(G)}$ by a standard argument.
	\end{proof}
	We also have the following consequences, which will be used in the proofs below.
	
	\begin{corollary}\label{prop:PreludeGInvIdeal}
		Take notations as in \Cref{prop:bch} and let $I$ be a strong ideal, then for any $y\in I$ and any $g\in \mf g$, it holds that $g\in G$ if and only if $y+g\in G$.
	\end{corollary}
	\begin{proof}
		It suffices to show one direction of the corollary. Let $g\in \mf g,y\in I$ be arbitrary, then $g+y\in g+I$ thus \cref{prop:CosetsEquivalent} implies that there exists some $y'\in I\subset G$ such that $g+y=g*y'$. As $G$ is a subgroup of $(\mf g,*)$, the result follows.
	\end{proof}
	\begin{corollary}\label{prop:GinvIdeal}
		Take notations as in \Cref{prop:bch}, then for any $g\in G,y\in\mf g$ holds that $g+my\in G$.
	\end{corollary}
	\begin{proof}
		As $m\mf g$ is a strong ideal, the result follows using \cref{prop:PreludeGInvIdeal}.
	\end{proof}

    The most important property of strong ideals that will be used is the relation between the Lie bracket and the commutator.
	
		\begin{proposition}\label{prop:ConjTranslation}
		Let $G,\mf g$ be as in \Cref{prop:bch}. Let $I$ be a strong ideal and let $c\in \mf g_{k}$. Then for $a,g\in\mf g$, we have that $\{a,g\}+I=c+I$ if and only if $[a,g]+I=c+I$.
	\end{proposition}
	\begin{proof}
		For the first implication, assume that $[a,g] + I = c + I$. By \cref{prop:bch} and the definition of strong ideal, we thus have $\{a,g\}+I=[a,g]+I$, which immediately gives us the result. 
		
        Conversely assume that $\{a,g\}\in c+I$. First we show by induction that whenever $i\leq k$, then $[a,g]\in\mf g_i+I$. For the base step, it is clear that $[a,g]\in \mf g_2$. For the induction step, assume that $[a,g]\in\mf g_i+I$. Then for any $v\neq(1)$, we have $r_v[a,g]_v\in \mf g_{i+1}+I$. By \cref{prop:bch} we have $$
		\{a,g\}+\mf g_{i+1}+I=[a,g]+\sum_v r_v[a,g]_v+\mf g_{i+1}+I
		$$
		or thus $\{a,g\} + \mf g_{i+1}+I=[a,g]+\mf g_{i+1}+I$. By induction it follows that $[a,g]\in\mf g_{k} + I$. Again we conclude that $[a,g]+I=\{a,g\}+I$.
	\end{proof}
	In particular the above result works for $I=0$.
	
	Although strong ideals satisfy stronger conditions than ideals, every ideal contains a strong ideal of small index.
	
	\begin{lemma}\label{prop:IdealContainsStrong}
	Take notations as in \Cref{prop:bch}. There exists a constant $C$ such that for any ideal $I \subset \mf g$, there exists some ideal $I'\subset I$ with $[I:I^\prime]\leq C $ and $I'$ is a strong ideal in $\mf g$. Furthermore, $I'$ can be chosen such that $m(\mf g_{k}\cap I)=\mf g_{k}\cap I'$ where $m$ is the constant from \cref{prop:bch}.
	\end{lemma}
	\begin{proof}
		Take $J = mI$, implying that $J \subset G$ by \cref{prop:bch}. Take $d$ the least common multiple of the denominators of the constants $q_v$ and $r_v$ for $v\in\mc I_{k}$. 
		Write $I_i = J \cap \mf g_i$ and take $I'=\sum_{i\leq k}d^{k-i}I_i$.
		Note that as both $I'\subset J$ and $\mf g_{k}\cap J\subset I'$, it holds that $\mf g_{k}\cap I'=\mf g_{k}\cap J=m(\mf g_{k}\cap I)$. The index of $[I:I^\prime]$ is uniformly bounded, as it is smaller than the index of $d^{k-1}I_1$ in $I$, which is bounded by $m^{h(\mf g)}d^{(k-1)h(\mf g)}$.
		
        It remains to show that $I^\prime$ is a strong ideal, or thus that if $x\in I'$ and $y\in \mf g$, then $q_v[x,y],r_v[x,y]\in I'$. By linearity of the maps $q_v[\_,y]$ and $r_v[\_,y]$ it suffices to show this whenever $x\in d^{k-i}I_i$ for some $i$. Take $x_0\in I_i$ with $x=d^{k-i}x_0$. Now $q_v[x,y]=q_vd(d^{k-i-1}[x_0,y])$. As $J$ is an ideal, it follows that $[x_0,y]\in J$ and thus $[x_0,y]\in I_{i+1}$. By definition of $d$, the number $q_vd$ is an integer and thus $q_v[x,y]\in (q_vd) d^{k-i-1}I_{i+1}\subset I'$. The same argument works for $r_v[x,y]$ and shows that $I^\prime$ is an ideal.
	\end{proof}
	Similarly, we can obtain a strong ideal from a normal subgroup of $G$.
\begin{lemma}\label{prop:SubgroupContainsStrong}
	Take notations as in \Cref{prop:bch}. There exists some constant $C$ such that for every normal subgroup $N \triangleleft G$, there exists a strong ideal $I$ in $G$ with $I\subset N$ and $[N:I] \leq C$. Furthermore, $I$ can be chosen such that $\mf g_{k}\cap N=\mf g_{k}\cap I$.
\end{lemma}
\begin{proof}
	Take $d$ the least common multiple of the denominators of the constants $q_v,r_v$ and $s_v$ for $v\in\mc I_{k}$ and write $I_i = N\cap \mf g_i$. Note that the $I_i$ are not automatically submodules. Define $I=\sum_{i\leq k}(md)^{k-i}I_i$, then we first show that $I$ is a strong ideal. For this, let $I^\prime_j=\sum_{j\leq i\leq k}(md)^{k-i}I_i$.
	We show by induction (where $j$ goes from $k$ to $1$) that for $x\in I_j'$ and $y\in \mf g$, the elements $[x,y],q_v[x,y],r_v[x,y],s_v[x,y]\in I^\prime_{j+1}$ and that $I_j'$ is a submodule of $\mf g$.
	
	Take any $x\in I_{j}'$. As $I_{j+1}'$ is by induction already a submodule, we might assume that $x\in (md)^{k-j} I_j$, so $x=(dm)^{k-j}  x_0$ for some $x_0\in I_j$. Let $\alpha \in \QQ$ be either $1,q_v,r_v$ or $s_v$ for some $v\in \mc I_{k}$.
	By \cref{prop:bchGroupCommutator}, we have$$
	\alpha d^{k-j} m^{k-j-1} \{x_0,my\}=\alpha[(dm)^{k-j} x_0,y]+\sum_{(1) \neq v \in \mc I_k}\alpha d^{k-j} m^{k-j-1}s_v\{x_0,my\}_v.
	$$
	As $\alpha d\in \ZZ$, we have that the left hand side is an element of $(dm)^{k-j-1} I_{j+1}$. Furthermore each of the repeated commutators $\{x_0,my\}_v$ is an element of $I_{j+2}$. As $I_{j+1}'$ is a submodule and $\alpha d^2s_v \in \ZZ$, we obtain that $$
	d^{k-j} m^{k-j-1} \alpha \{x_0,my\}-\sum_v d^{k-j} m^{k-j-1}s_v \alpha \{x_0,my\}_v\in I_{j+1}'.
	$$
	and thus that $\alpha [x,y]\in I_{j+1}'$.
	
    We now show that $I_j'=(dm)^{k-j} I_j+I_{j+1}'$ is closed under addition. Again it suffices to show for $x,y\in (dm)^{k-j} I_j$ that $x+y\in (dm)^{k-j} I_j+I_{j+1}'$. Let $x_0,y_0\in I_j$ such that $x=(dm)^{k-j} x_0$ and $y=(dm)^{k-j} y_0$
	We have by \cref{prop:bch}$$
	(dm)^{k-j} (x_0*y_0)=(dm)^{k-j} x_0+(dm)^{k-j} y_0+\sum_v (dm)^{k-j} q_v[x_0,y_0]_v
	$$
	The left hand side belongs to $(dm)^{k-j} I_j$ and the last term belongs to $I_{j+1}'$ by the above, thus $x + y = (dm)^{k-j} x_0+(dm)^{k-j} y_0$ belongs to $(dm)^{k-j} I_j+I_{j+1}'$. By induction, $I=I_1'$ is a submodule and satisfies \[ [I,\mf g],\ q_v[I,\mf g],\ r_v[I,\mf g]\subseteq I. \] Hence $I$ is a strong ideal.
    
   Since $(md)^{k-j}I_j\subseteq I_j'$ for every $j$, the successive quotients along the adaptive lower central series give \[ [N:I]\leq\prod_{j=1}^{k-1}(md)^{(k-j)h(\mf g_j/\mf g_{j+1})}\leq(md)^{(k-1)h(\mf g)}. \] Thus the result follows with $C=(md)^{(k-1)h(\mf g)}$.
\end{proof}	
	

	\subsection{Conjugacy in Lie rings}

	\cref{prop:ConjTranslation} gives us a strategy to demonstrate \cref{prop:AlgUpperBound} by first reducing the problem to separating elements $a$ and $ac$ where $c\in G_{k}$.
	These elements are conjugate if and only if there exists some $g\in G$ such that $\{a,g\}=c$ or thus by the above such that $[a,g]=c$. We then find an ideal $I$ of $\mf g$ such that $c\notin [a,G]+I$. Finally, we translate $I$ to a normal subgroup of $G$ such that $c\notin \{a,G\}I$. For what follows, we fix a generating set $S$ of $G$ and a basis $S^\prime$ for $\mf g$.
	\newcommand{\normg}[1]{\norm{#1}_{S}}
	\newcommand{\normr}[1]{\norm{#1}_{+,S^\prime}}
	To avoid confusion, we denote the word norm in $G$ with $\norm{\cdot}_{* S}$ and the one in $\mf g$ by $\norm{\cdot}_{+,S^\prime}$.

	In the following lemmas, we determine the index of $[a,\mf g]$ in its isolator.
	\begin{lemma}\label{prop:NormCorespondence}
		Take notations as above, then there exists some constant $C$ such that for any $g\in G$ and any $i$, the norm $\normr{g+\mf g_{i+1}}$ in $\frac{\mf g}{\mf g_{i+1}}$ can be bounded above by $\normr{g+\mf g_{i+1}}\leq C\normg{g}^{ i}$.
	\end{lemma}
	\begin{proof}
		By \cite{guivarch1973croissance}, the norm $\normg{ \cdot}$ is bilipschitz to the Guivarch norm $\gnorm{\cdot}$. The statement follows from the definition of the Guivarch norm.
	\end{proof}
	\begin{lemma}\label{prop:NormOfCommut}
		For any $a\in G$ the linear map $[a,\_]$ has operator norm (with respect to $\normr{\_}$) at most $C\normg{a}^{k-1}$ for some uniform constant $C$.
	\end{lemma}
	\begin{proof}
		By \cref{prop:NormCorespondence}, $a+\mf g_{k}$ can be written as a sum $$
		\sum_{s\in S^\prime}a_ss+\mf g_{k}
		$$
		where $\sum_{s \in S^\prime} \vert a_s\vert \leq C'\normg{a}^{k-1}$ for some uniform constant $C'$. As $\mf g$ is $k$-step nilpotent, we have that the linear map $[a,\_]$ can thus be rewritten as $$
		\sum_{s\in S^\prime}a_s[s,\_].
		$$
		Let $C''$ be the maximum of the norms of the maps $[s,\_]$ where $s\in S^\prime$, then it follows that the norm of $[a,\_]$ is bounded above by $$\sum_{s\in S^\prime} \vert a_s \vert C''\leq C''C\normg{a}^{k-1}$$
	\end{proof}
	\begin{corollary}\label{prop:DetComutBound}
		Take notations as in \Cref{prop:bch}, then there exists some constant $C$ such that for $a\in G$ arbitrary holds that $\det([a,\mf g])\leq C\normg{a}^{(k-1)h_2(G)}$
	\end{corollary}
	\begin{proof}
		This follows by combining \cref{prop:NormOfCommut} with \cref{prop:SubspaceNormToIndex}.
	\end{proof}
	\cref{prop:ConjTranslation} suggests that we need to separate the conjugacy classes of $a$ and $b$ when $a \ast G_{k}=b \ast G_{k}$. The following lemmas help us reduce to this case.

	\begin{lemma}\label{ACReduction}
		Let $G$ be a torsion-free finitely generated nilpotent group of nilpotency class $k$ with generating set $S$. Then there exists some polynomial $p$ such that whenever $a,b\in G$ are such that $\normg{a},\normg{b}\leq n$ and such that $a G_{k}\sim b G_{k}$, then there exists some $c\in G_k$ such that $\normg{c}\leq p(n)$ and such that $ac\sim b$.
	\end{lemma}
	\begin{proof}

	By \cite[Lemma 6.2]{dere2019effective} with $\varphi$ conjugation with $a$ on $\frac{G}{G_k}$, there exists some $gG_{k}\in \frac{G}{G_{k}}$ such that $aG_{k}=gbg\inv G_{k}$ and such that $\normg {gG_{k}}\leq q(n)$ for some polynomial $q(n)$ not depending on $g,a,b,n$. Take now $g$ a representative of $gG_{k}$ such that $\normg{g}=\normg{gG_{k}}$. Take then $c= gbg\inv a\inv$. Then $c$ has the desired properties for $p=2n+2q(n)$.
	\end{proof}
	The final ingredient before we demonstrate the upper bound is the following.
	\begin{lemma}\label{prop:MinQuotientCyclicCenter}
		Let $G$ be a group such that $Z(G)$ is finitely generated and let $a\in G,c\in Z(G)$ such that $a\not\sim ac$, then there exists a quotient $\varphi:G\rightarrow Q$ such that $\varphi(a)\not\sim \varphi(ac)$ and such that $\varphi(Z(G))$ is cyclic.
	\end{lemma}
	\begin{proof}
		Let $[a,G]$ denote the set $\{[a,g]\mid g\in G\}$. Notice that $[a,G]\cap Z(G)$ is a central subgroup, and thus normal. Consider the finitely generated abelian group $A = \frac{Z(G)}{Z(G)\cap[a,G]}$. As $c$ is a non-trivial element in this group, it follows by the classification of finitely generated abelian groups that there exists a subgroup $N\normal A$ such that $c$ is non-trivial in the cyclic group $A/N$. The inverse image $\tilde{N}$ of $N$ in $Z(G)$ is a normal subgroup of $G$, and hence  the quotient $Q=\frac{G}{\tilde{N}}$ has the desired properties.
	\end{proof}
	We are now ready to prove \cref{prop:AlgUpperBound}
	\begin{proof}[\cref{prop:AlgUpperBound}]
		Let $a,b\in G$ be non-conjugate. We may assume that $aG_{k}\sim bG_{k}$, indeed if this is not the case, we may replace $G$ by $\frac{G}{G_{i+1}}$ where $i$ is minimal such that $aG_{i}\sim bG_{i}$.
		By \cref{ACReduction}, we have that $b$ is conjugate to $ac$ for some $c\in G_{k}$ where $\normg{c}\leq p\left(\max\set{\normg{a},\normg{b}}\right)$ for some polynomial $p$, independent of $a$ and $b$.
		As $a$ and $b$ are non-conjugate, it follows that $a$ and$ac$ are also non-conjugate and thus that $c$ is not of the form $\{a,g\}$ for $g\in G$. By \cref{prop:ConjTranslation} we know that $c\neq[a,g]$ for all $g \in G$.
		
        There are now two possibilities, depending on whether $c$ lies in $[a,\mf g]$. In both cases we show that there exists some ideal $I = q \mf g$ such that $c\notin [a,G]+I$.
        
	For the first case, assume that $c \in [a,\mf g]$ and fix $g_0\in \mf g\backslash G$ with $c=[a,g_0]$. Let $m$ be the constant from \cref{prop:GinvIdeal} and let $d$ be the determinant of $[a,\mf g]$. We claim that $q = md$ works. Indeed, if there would be a $g$ such that $[a,g] \in c + dm \mf g$, then $[a,g-g_0]\in dm\mf g$. As $\sqrt{[a,\mf g]}\cap d\mf g\subset [a,\mf g]$, we have that $[a,g-g_0]\in [a,m\mf g]$ or thus that there exists some $g'\in \mf g$ with $[a,g^\prime] = 0$ such that $g-(g_0+g')\in m\mf g \subset G$. It follows by \cref{prop:GinvIdeal} that $g_0+g'\in G$. But $[a,g_0+g']=[a,g_0]+[a,g']=c+0$, contradicting the fact that no $g\in G$ exists such that $[a,g]=c$.
		
		For the second case, by \cref{prop:AbelianSubgroupSeparability} there exists a quotient of $\ZZ$-modules $\varphi:\mf g\rightarrow Q$ such that $\varphi(c)\notin\varphi([a,\mf g])$ and such that $Q$ is bounded above by $$
		\abs{Q}\leq\max\set{\det([a,\mf g]), C\log(p(\max\set{\normg{a},\normg{b}}))}.
		$$
		where $C$ is some constant only depending on $\mf g$.
The kernel contains the ideal $q\mf g$ where $$
		q\leq\max\set{m\det([a,\mf g]), Cm\log(p(\max\set{\normg{a},\normg{b}}))},
		$$ and thus the ideal $I = q \mf g$ satisfies the claim. Here the $m$ is added to make it a subset of $G$.

	By \cref{prop:IdealContainsStrong}  there exists a strong ideal $I^\prime \subset I$ with $[\mf g:I']\leq C_0[\mf g:I]$ for some uniform constant $C_0$ and  $$
		m\left(\mf g_{k}\cap I\right)=\mf g_{k}\cap I'
		$$
		where $m$ is the constant of \cref{prop:bch}.
		By \cref{prop:ConjTranslation}, it follows that there is no $g\in G$ such that $\{a,g\}=c+I'$. By \cref{prop:CosetsEquivalent} it thus follows that $\{a,g\} *I^\prime \neq c*I'$ for all $g \in G$, or hence $a*I'$ and $a*c*I'$ are non-conjugate.
		
		To bound the index of $I'$, note that $$
		C_0q^{h(G)}\leq C_0 \max\set{m\det([a,\mf g]), C m \log(p(\max\set{\normg{a},\normg{b}}))}^{h(G)}
		$$
		and $Z(G)/(Z(G)\cap I')$ is isomorphic to a subgroup of $\frac{\ZZ^{h(Z(G))}}{mq\ZZ^{h(Z(G))}}$ containing $\frac{m\ZZ^{h(Z(G))}}{mq\ZZ^{h(Z(G))}}$.
		By \cref{prop:MinQuotientCyclicCenter} it follows that there also exists normal subgroup $I''\supset I^\prime$ such that $a*I''$ and $b*I''$ are non-conjugate and such that $\frac{Z(G)}{Z(G)\cap I''}$ is cyclic. It follows that index of $I''$ can be further reduced to be bounded by $(mq)^{h(G)-h(Z(G))+1}$. The result follows by applying \cref{prop:DetComutBound}, leading to $$[G:I'']\preceq n^{(k-1)(h(G)-h(Z(G))+1)h(G_2)}$$ as logarithmic terms are always bounded by polynomials.

	\end{proof}
	\cref{sec:AlgBoundScharp} will demonstrate that the exponent in \cref{prop:AlgUpperBound} must be at least quadratic in the Hirsch length. It is unclear, however, how far this theorem is from being sharp in case of higher nilpotency classes.
	\begin{question}
		Is there a constant $C$, not dependent on the nilpotency class, such that for any torsion-free finitely generated nilpotent group $G$, we have $\conj_G(n)\preceq n^{C h(G)^2}$?
	\end{question}

	\section{Two-step nilpotent groups}\label{sec:groupsandLierings}

The main goal of this section is to show that, for torsion-free $2$-step nilpotent groups, conjugacy separability growth is essentially a Lie-ring invariant. This is in sharp contrast with the higher-step case, where the correspondence between groups and Lie rings only yields partial information as demonstrated above. By translating the problem completely into the setting of Lie rings, we obtain stronger upper bounds and are able to study their sharpness.
        

	
	\subsection{Conjugacy separability in Lie rings}

    In general, the methods from the previous section can only be used to produce upper bounds to $\conj_{G}$. However when the map $\log$ from \cref{prop:bch} is surjective, so when $(G,*)=(\mf g,*)$, then the relation is stronger. First we introduce a corresponding function on Lie rings.
	\begin{definition}\label{sepinLiering}
		Let $\mf g$ be a finitely generated torsion-free nilpotent Lie ring. For any $a,c\in \mf g$ with $c\notin [a,\mf g]$, define $\cd_{\mf g}$ as $$
		\cd_{\mf g}(a,c)=\min\{[\mf g:I]\mid I\normal \mf g,c\notin [a,\mf g]+I\}.
		$$  
	Define $\conj_\mf g: \NN \to \NN$ as $$
		\conj_{\mf g}(n)=\max\{\cd_{\mf g}(a,c)\mid c\notin [a,\mf g], \norm{a}_{\mathrm{g}},\norm{c}_{\mathrm{g}}\leq n\}
		$$
        for any $n \in \NN$.
	\end{definition}
	As before, the function $\conj_{\mf g}$ depends on the choice of Guivarch norm. However, as all these norms are bilipschitz, the resulting function $\conj_{\mf g}$ is  well-defined up to $\simeq$.

	When $G$ corresponds to a Lie ring, we have the following estimates.
	\begin{proposition}\label{prop:CdGroupToRing}
		Let $\mf g$ be a finitely generated $k$-step nilpotent Lie ring that is closed under $\ast$. For every $a\in \mf g$ and $c\in \mf g_{k}$ it holds that $ac\sim a$ if and only if $c\in [a,\mf g]$. Furthermore, if $ac\nsim a$, then $$
		\cd_{(\mf g,*)}(a,ac)\leq C\cd_{\mf g}(a,c)
		$$
		and
		$$
		\cd_{\mf g}(a,c)\leq C\cd_{(\mf g,*)}(a,ac)
		$$
		where $C$ is a uniform constant depending only on $\mf g$.
	\end{proposition}
	\begin{proof}
		Notice that $ac$ and $a$ are conjugate precisely when there exists some $g$ such that $c=\{a,g\}$. By \cref{prop:ConjTranslation} for the ideal $0$, the first part follows. 
        
        Let $C$ be the maximum of the constants from \cref{prop:IdealContainsStrong} and \cref{prop:SubgroupContainsStrong}.
		For the first inequality, let $I\normal \mf g$ be an ideal with $c\notin [a,\mf g]+I$ and $[\mf g: I] = \cd_{\mf g}(a,c)$. By \cref{prop:IdealContainsStrong} there exists a strong ideal $I' \subset I$  with $[\mf g : I^\prime] \leq C\cd_{\mf g}(a,c)$. This strong ideal is a normal subgroup by \cref{prop:StrongIdealIsNormal} and $c \notin \{a,\mf g\}+I'$ by \cref{prop:ConjTranslation}. which by \cref{prop:CosetsEquivalent} is equal to $\{a,\mf g\} \ast I'$. In particular $a \ast I'\nsim ac \ast I'$, where by \cref{prop:IndexCorrespondence} $I'$ as a normal subgroup is of index at most $C\cd_{\mf g}(a,c)$.
		
		The other direction is almost identical using \cref{prop:SubgroupContainsStrong}.
	\end{proof}

	The correspondence between groups and Lie rings is stronger in the $2$-step case because we can always find a group with the same conjugacy growth function where the map $\log$ from \cref{prop:bch} is surjective.
	\begin{proposition}\label{prop:centralextension}
		Let $G\normal H$ be finitely generated torsion-free nilpotent groups. Suppose that $\frac{H}{G}$ is finite and that $H=GZ(H)$, then $\conj_{G}\simeq \conj_{H}$.
	\end{proposition}
	\begin{proof}
		Fix generating sets $S_1, S_2$ for the groups $G, H$. As the groups $G$ and $H$ are commensurable, there exists some constant $C>0$ such that for all $g\in G$, $$
		\frac{1}{C}\norm{g}_{S_1}\leq \norm{g}_{S_2}\leq C\norm{g}_{S_1}.
		$$ 
		
		To show that $\conj_{G}\preceq \conj_{H}$, take $g_1,g_2\in G$ that are non-conjugate with $\norm{g_1}_{S_1},\norm{g_2}_{S_1}\leq n$. Notice that as conjugating with elements of $Z(H)$ leaves $g_1,g_2$ invariant, $g_1$ and $g_2$ are also non-conjugate in $H$. Let $\varphi:H\rightarrow Q$ be a finite quotient of size $\cd_{H}(g_1,g_2)$ such that $\varphi(g_1)\nsim\varphi(g_2)$, then $\varphi$ restricts to a quotient $\varphi\mid_{G}:G\rightarrow \varphi(G)$ of size bounded by $\cd_{H}(g_1,g_2)$. It still holds that $\varphi(g_1)\nsim\varphi(g_2)$ in $\varphi(G)$ and thus $\cd_{G}(g_1,g_2)\leq \cd_{H}(g_1,g_2)\leq \conj_{H,S_2}(Cn)$, leading to the statement.
		
		For the other bound, take $g_1,g_2\in H$ non-conjugate with $\norm{g_1}_{S_2},\norm{g_2}_{S_2}\leq n$. We can always assume that $g_1\in G$. Indeed there exists some $c\in Z(H)$ such that $g_1c\in G$. Furthermore, $g_1N\sim g_2N$ if and only if $g_1cN\sim g_2cN$ for any normal subgroup $N$. Notice that such a $c$ can be chosen from a fixed finite set of representatives of $\frac{H}{G}$, and thus can $c$ be chosen of uniformly bounded norm. 
		
		We now distinguish $2$ cases, namely $g_2\notin G$ or $g_2\in G$. In the first case, consider $\varphi:H\rightarrow \frac{H}{G}$. As $\varphi(g_1)$ is trivial and $\varphi(g_2)$ is not, they are non-conjugate. In this case we thus have $\cd_{H}(g_1,g_2)\leq \abs{\frac{H}{G}}$. For the other case, take $g_2\in G$ and let $\varphi:G\rightarrow Q$ be a finite quotient with kernel $N$ such that $\varphi(g_1)\nsim\varphi(g_2)$. Such a quotient can be chosen of size at most $\cd_{G}(g_1,g_2)$.
		Since $H = G Z(H)$, we have that $N$ is also a normal subgroup of $H$. Denote $\varphi':H\rightarrow \frac{H}{N}$. Then $Q$ is a subgroup of $\frac{H}{N}$ and $\frac{H}{N}=QZ(\frac{H}{N})$. In particular, $\varphi(g_1)$ and $\varphi(g_2)$ are conjugate in $Q$ if and only if they are in $\frac{H}{N}$. The index of $[H:N]$ is given by $[H:G][G:N]$ and thus is $\cd_{H}(g_1,g_2)\leq \abs{\frac{H}{G}}\cd_{G}(g_1,g_2)\leq \abs{\frac{H}{G}}\conj_{G,S_2}(Cn)$.
	\end{proof}
	\begin{lemma}\label{prop:ExistsLogSurjective}
		Let $G$ be a torsion-free $2$-step nilpotent group, then there exists some Lie ring $\mf g$ satisfying \cref{prop:bch} such that $\mf g=G*Z(\mf g)$.
	\end{lemma}
	\begin{proof}
		Let $\mf g'$ be any Lie ring satisfying \cref{prop:bch} with corresponding map $\log: G\rightarrow \mf g'$. Take $\mf g=\log(G)+\frac{1}{2}[\log(G),\log(G)]$, then $\mf g$ is clearly a Lie subring of $\mf g'$ that is closed under $\ast$. The final condition is immediate from the construction.
	\end{proof}

	\begin{proposition}\label{prop:CSepGrowthEquiv}
		Let $G$ be a finitely generated torsion-free 2-step nilpotent group with $\mf g$ as in \cref{prop:bch} such that $\log:G\rightarrow\mf g$ is surjective, then $\conj_G\simeq\conj_\mf g$
	\end{proposition}
	\begin{proof}
		As usually we identify $G$ with $\log(G)$. First, note that both of these functions grow at least as $\log$, as they are bounded from below by the residual finiteness growth of their subgroups, and thus by $\mathrm{Rf}_{\ZZ}\simeq \log$.
		Furthermore, by \cite{guivarch1973croissance}, there exists a constant $C$ such that for any $g\in G$, $\frac{1}{C}\norm{g}_S\leq \norm{g}_{\mathrm{g}}\leq C\norm{g}_S$.
		
		Let $g_1,g_2\in G$ be non-conjugate such that $\norm{g_1}_S,\norm{g_2}_S\leq n$. We distinguish $2$ cases, first assume that $g_1G_2\neq g_2G_2$. In this case we find an abelian quotient (of groups) $\varphi: G\rightarrow Q$ of size at most $\mathrm{Rf}_{\frac{G}{G_2}}(2n)$ such that $\varphi(g_1)\neq \varphi(g_2)$. As $Q$ is abelian, we also have that $\varphi(g_1)\nsim\varphi(g_2)$. On the other hand suppose that $g_1G_2=g_2G_2$, then $g_2g_1\inv$ is an element of $G_2$ of norm at most $2n$. We thus have by \Cref{prop:CdGroupToRing} that $\cd_{G}(g_1,g_2)\leq C'\cd_{\mf g}(g_1,g_2g_1\inv)$ and thus $\cd_{G}(g_1,g_2)\leq C'\conj_{\mf g}(2Cn)$. Combining the two cases gives that $\conj_G\prec\conj_{\mf g}$.
			
		The other direction is mostly analogous, let $a,c\in \mf g$ be such that $c\notin[a,\mf g]$. Again we distinguish $2$ cases. In the first case, suppose that $c\notin\mf g_2$, then $c$ does not vanish in some finite abelian quotient of $\frac{\mf g}{\mf g_2}$ of size at most $\mathrm{Rf}_{\frac{\mf g}{\mf g_2}}(n)$. In the other case, we again use \Cref{prop:CdGroupToRing} to obtain $\cd_{\mf g}(a,c)\leq C'\conj_{G}(2Cn)$.
	\end{proof}
	The above gives us \cref{prop:conjequal}. Here the ``corresponding'' Lie ring is the Lie ring as in \cref{prop:ExistsLogSurjective}. This Lie ring has the same conjugacy separability by \cref{prop:CSepGrowthEquiv} and \cref{prop:centralextension}. 

	\subsection{Retractions and products}
Conjugacy does not always behave well with respect to taking subgroups. In this section we demonstrate that the situation for retracts is well-behaved. Although this works for groups as well, we only work in the language of Lie rings.
	\begin{definition}
		Let $\mf g$ and $\mf h$ be Lie rings, with $i:\mf h\rightarrow \mf g$ a morphism. A \textbf{retraction} $r: \mf g \to \mf h$ is a left inverse of $i$, so satisfying $r\comp i=\id_{\mf h}$. If such a retraction exists, then we say that $\mf h$ is a retract of $\mf g$.
	\end{definition}
    	As retracts have a left inverse, they are always injective, hence $\mf h$ can be seen as a Lie subring of $\mf g$. Dually,  the retraction map has a right inverse and must thus be surjective.
	\begin{example}
		The embedding of a Lie ring $\mf h$ into the product $\mf g \times \mf h$ is a retract, with the retraction given by the projection on the second component. 
	\end{example}
	The following lemma shows that retracts behave well for conjugacy separability.
	\begin{lemma}
		\label{prop:retractlowerbound}
		Let $\mf g,\mf h$ be finitely generated nilpotent Lie rings and let $i:\mf h\rightarrow\mf g$ be a retract with retraction map $r$. For all $a,c\in\mf h$ such that $c\notin[a,\mf h]$, it holds that $\cd_{\mf h}(a,c)=\cd_{\mf g}(i(a),i(c))$. Furthermore, $\conj_{\mf h}\preceq \conj_{\mf g}$.
	\end{lemma}
	\begin{proof}
		First, we demonstrate that $i(c)\notin[i(a),\mf g]$ and thus that $\cd_{\mf g}(i(a),i(c))$ is defined. Assume for a contradiction that $i(c)\in[i(a),\mf g]$, then there would exist some $g\in\mf g$ such that $i(c)=[i(a),g]$. Applying the retraction $r$ gives then $c=[a,r(g)]$, which contradicts $c\notin[a,\mf h]$. 
        
		Next we show that $\cd_{\mf h}(a,c)=\cd_{\mf g}(i(a),i(c))$. Let $I$ be an ideal in $\mf h$ such that $c\notin[a,\mf h]+I$.  Take $J = r\inv (I)$, then just as before $i(c)\notin[i(a),\mf g]+J$. As $r$ is surjective, we have that $[\mf g : J] = [\mf h: I]$ and thus $\cd_{\mf h}(a,c)\geq \cd_{\mf g}(i(a),i(c))$.
		
		For the other inequality, take $J$ be an ideal in $\mf g$ such that $i(c)\notin[i(a),\mf g]+J$. Consider now $I=i\inv (J)$. As $i$ is injective, we know $[\mf h :I] \leq [\mf g: J]$ and moreover $c\notin[a,\mf h]+I$. Indeed, if the latter were not the case, then there would exist some $h\in\mf h$ and $h'\in I$ such that $c=[a,h]+h'$ and thus also such that $i(c)=[i(a),i(h)]+i(h')$, which contradicts the definition of $J$. This demonstrates that $\cd_{\mf h}(a,c)\leq \cd_{\mf g}(i(a),i(c))$.
		
        The inequality $\conj_{\mf h}\preceq \conj_{\mf g}$ follows as the Guivarch norm decreases when taking morphisms.
	\end{proof}
	
	\begin{lemma}
		\label{prop:retractupperbound}
		Let $\mf g,\mf h$ be $2$-step nilpotent finitely generated Lie ring, and let $i:\mf h\rightarrow\mf g$ be a retract with retraction map $r$. For any $a\in\mf g$ and $c\in\gamma_2\mf g$ with $c\notin[a,\mf g]$ and $r(c)\notin [r(a),\mf h]$, it holds that $\cd_{\mf g}(a,c)\leq \cd_{\mf h}(r(a),r(c))$.
	\end{lemma}
	\begin{proof}
		Let $I \triangleleft \mf h$ such that $r(c)\notin[r(a),\mf h]+I$, then take $J = r\inv I$. As $r$ is surjective $[\mf g : J] = [\mf h: I]$ and moreover $c\notin[a,\mf g]+J$, just as in the previous proof. 
	\end{proof}

	In the special case of the direct product, we thus obtain the following result:
	\begin{corollary}\label{prop:directSumDecomposition}
		Let $\mf g,\mf h$ be $2$-step nilpotent Lie rings, then $\conj_{\mf g\times\mf h}\simeq\max(\conj_{\mf g},\conj_{\mf h})$
	\end{corollary}
	\begin{proof}
		By \cref{prop:retractlowerbound}, we have $\max(\conj_{\mf g},\conj_{\mf h})\preceq \conj_{\mf g\times\mf h}$. Let $(a_1,a_2),(c_1,c_2)\in \mf g\times\mf h$ such that $(c_1,c_2)\notin[(a_1,a_2),\mf g\times\mf h]=[a_1,\mf g]\times [a_2,\mf h]$. By symmetry we can assume that $c_1\notin[a_1,\mf g]$. Then by \cref{prop:retractupperbound} we have that $\cd_{\mf g\times\mf h}((a_1,a_2),(c_1,c_2))\leq \cd_\mf g(a_1,c_1)$. The result follows.
	\end{proof}

	\subsection{Bounding $\conj$ in the $2$-step Case}\label{sec:AlgUpperBound}
	Throughout this section we will assume that $G$ is a torsion-free $2$-step nilpotent group.
	By \cref{prop:CSepGrowthEquiv} we know that $\conj_G\simeq\conj_{\mf g}$ where $\mf g$ is an associated Lie ring. \cref{prop:AlgUpperBound2} is thus equivalent to:
	\begin{theorem}\label{prop:AlgUpperBoundLieRing}
		Let $\mf g$ be a $2$-step nilpotent Lie ring. Let $h$ be the rank of $\mf g$, $h_2$ the rank of $\mf g_2$ and $h_z$ the rank of $Z(\mf g)$. Suppose that $h_2\geq 2$, then $$\conj_{\mf g}(n)\preceq n^{h_2\lfloor h-h_z\rfloor_{\mathrm{odd}}}.$$
	\end{theorem}

    Indeed, the case of $h_2=1$ was already done in \cite{pengitore2015effective}, where the author showed that $\conj_G(n)\simeq n^{h-h_z+1}$. Notice that in this case the above bound does not hold and hence the condition $h_2>1$ is necessary. In all other cases, $h{\mf g}\geq h{Z(\mf g)}+3$. Otherwise, the quotient $\frac{\mf g}{Z(\mf g)}$ would be either cyclic, and by consequence the Lie bracket would be trivial, or the quotient would be of rank $2$, and thus the commutator subgroup of rank $1$. Hence $h(Z(\mf g))(\lfloor h(\mf g)-h(Z(\mf g))\rfloor_{\mathrm{odd}})\geq h(\mf g)+2$.

	From now on, fix a basis $S$ on $\mf g$. We will be using the following lemmas:
	
	\begin{lemma}\label{prop:LieBracketBounded}
		There exists some constant $C>0$ such that for any $a\in \mf g$, and for any $s\in S$, $$
		\norm{[a,s]}_{S}\leq C\norm{a}_S
		$$
	\end{lemma}
	\begin{proof}
		By definition of a Lie ring, the map $[\_,s]$ is a linear map from $\mf g$ to $\mf g_2$. Linear maps between finite-dimensional normed spaces are always bounded, thus there exists some constant $C_s\geq0$ such that for any $a\in \mf g$, $$
		\norm{[a,s]}_{S}\leq C_s \norm{a}_S.
		$$
		As $S$ is a finite set, the result follows by taking $C=\max\set{C_s\mid s\in S}$.
	\end{proof}
	
	\begin{lemma}\label{prop:DimensionSymplecticForm}
		Let $\mf g$ be a torsion-free Lie ring with $\mf g_2$ infinite cyclic. Let $Z(\mf g)$ denote the center of $\mf g$. Then $h(\mf g)-h(Z(\mf g))$ is even.
	\end{lemma}
	\begin{proof}
		Let $\{g_1,\cdots g_n\}$ be a basis for $Z(\mf g)^\perp$, a complement of $Z(\mf g)$. Since $\mf g_2$ is infinite cyclic, the Lie bracket induces an alternating bilinear form \[ Z(\mf g)^\perp\times Z(\mf g)^\perp\to \mf g_2 \approx \ZZ. \] This map is non-degenerate, otherwise there would exist some $0\neq g\in Z(\mf g)^\perp$ for which $[g,\mf g]=0$. In particular the rank of $Z(\mf g)^\perp$ must be even.
	\end{proof}

Let $Z(\mf g)^\perp$ be a complement of $Z(\mf g)$ in $\mf g$. As the center $Z(\mf g)$ is isolated, we have that $\mf g=Z(\mf g)\oplus Z(\mf g)^\perp$. We first prove a special case of \cref{prop:AlgUpperBoundLieRing}.
	\begin{proposition}\label{prop:AlgUpperBoundIndex1}
		Let $\mf g=Z(\mf g)\oplus Z(\mf g)^\perp$ be a $2$-step nilpotent Lie ring with basis $S$.  Then for $a\in Z(\mf g)^\perp$ and $c\in Z(\mf g)\backslash [a,\mf g]$ such that $\ind a=1$, there exists some quotient $\varphi:\mf g\rightarrow \mf q$ such that $\varphi(c)\notin\varphi ([a,\mf g])$.
		This quotient can be chosen such that $$\abs{\mf q}\leq \const{0}\max\set{\norm{a}_S,\log(\max\set{\norm{a}_S,\norm{c}_S})}^{h(Z(\mf g))\lfloor h(Z(\mf g)^\perp)\rfloor_{\mathrm{odd}}}$$ where $\const 0$ only depends on the choice of basis $S$.
	\end{proposition}
	\begin{proof}
		By \cref{prop:LieBracketBounded}, $\norm{[a,\mf g]}_{S}$ is bounded above by $\const{1}\norm{a}_{S}$ for some $C_1  > 0$. By \cref{prop:AbelianSubgroupSeparability} there exists some finite quotient $\varphi_0:Z(\mf g)\rightarrow Q$ of $\ZZ$-modules such that $\varphi_0(c)\notin \varphi_0([a,\mf g])=0$ and such that $${\abs{Q}\leq \max\set{(\const{1}\norm{a})^{h_z},\const{2}\log^{h_z}(\max\set{\const 1\norm{a}_{S_a},\norm{c}_{S_z}})}}.$$ Without loss of generality, we may assume that $Q$ is cyclic, projecting on one of its cyclic summands otherwise. Denote the order of $Q$ with $q$ and denote the kernel of $\varphi_0$ by $K_0$.

		As $Q$ is cyclic, there exists some basis $\set{v_i}$ of $Z(\mf g)$ such that $K_0 = \ZZ v_1+\ZZ v_2+\cdots +q\ZZ v_l$. Take  $K_1=\ZZ v_1+\ZZ v_2+\cdots+\ZZ v_{l-1}$, and note that the commutator restricts to a map $$
		[\_,\_]:(\langle a \rangle+Z(\mf g))^\perp\times(\langle a \rangle+Z(\mf g))^\perp\rightarrow \frac{Z(\mf g)}{K_1}.
		$$
		If $Z(\mf g)^\perp$ is odd-dimensional, then we define the element $x_0 = 0$. If $Z(\mf g)^\perp$ is even-dimensional, then $(\langle a \rangle+Z(\mf g))^\perp$ must be odd dimensional. By \cref{prop:DimensionSymplecticForm} it follows in this case that there exists some $x_0\in (\langle a \rangle+Z(\mf g))^\perp$ such that $[x_0,(\langle a \rangle+Z(\mf g))^\perp]\subset K_1$. As $\sqrt{K_1} = K_1$, dividing $x_0$ by a constant if necessary, we may assume $x_0$ to be of index $1$.
		
		We show that $K=K_0 +  qZ(\mf g)^\perp+\langle a \rangle+\langle x_0\rangle$ is an ideal in $\mf g$. Clearly this space is a submodule, thus it suffices to show that $[K,\mf g] \subset K$. Take any $g\in\mf g$ and $k_0+k_1+l_aa+l_xx_0 \in K$ with $k_0\in K_0, k_1\in qZ(\mf g)^\perp$ and $l_a,l_x\in\ZZ$. As $[k_0,g] = 0$, $[k_1,g] \in q Z(\mf g) \subset K_0$ and $[l_aa,g] \in [a,\mf g] \subset K_0$, it suffices to consider the last term $[l_x x_0,g]$. Note that $[x_0,g] = [x_0,g'+l_a'a]$ with $g'\in (\langle a \rangle+Z(\mf g))^\perp$. Now by definition of $x_0$ it holds that $[x_0,g']\in K_1\subset K_0$ and $[x_0,l_a'a] \in [a,\mf g] \subset K_0$. Combining these shows that $[k_0+k_1+l_aa+l_xx_0,g]\in K$, demonstrating that $K$ is indeed an ideal.
		
		We show that $c+K\notin [a,\mf g]+K$. Indeed, otherwise there would exist some $k\in K$ such that $c+k\in[a,\mf g]$. As both $c$ and $[a,\mf g]$ lie in $Z(\mf g)$, so must $k$, and thus is $k\in K_0$. But this would imply that $c+K_0\in [a,\mf g]+K_0$ or thus $\varphi_0(c)\in\varphi_0([a,\mf g])$, contradicting the previous.
		
		We compute the index of $K$. This index is the product of the index of $K_0$ in $Z(\mf g)$ and the index of $qZ(\mf g)^\perp+\langle a\rangle+\langle x_0\rangle$ in $Z(\mf g)^\perp$. Write $U = \langle a\rangle+\langle x_0\rangle$, As $a$ and $x_0$ are both of index $1$, and as $x_0\in(\langle a\rangle+Z(\mf g))^\perp$ we have that $U$ is isolated. It follows that $Z(\mf g)^\perp = U \oplus U^\perp$ for some complement $U^\perp$. Then $qZ(\mf g)^\perp+U = U + q U^\perp$. If $h(Z(\mf g)^\perp)$ is odd, then $x_0$ is trivial and then this subspace has index $q^{h(Z(\mf g)^\perp)-1}$. Otherwise this submodule has index $q^{h(Z(\mf g)^\perp)-2}$.
		The result follows as
		$$ q\leq\max\set{(\const 1\norm{a})^{h_z},\const2\log^{h_z}\left(\max\set{\const 1\norm{a}_{S_a},\norm{c}_{S_z}}\right)}.$$
	\end{proof}
	We can now conclude the proof of \cref{prop:AlgUpperBound2}
	\begin{proof}[Proof of \cref{prop:AlgUpperBoundLieRing}]
		We may assume that $\mf g_2=Z(\mf g)$, if this is not the case then $\mf g$ is of the form $\mf g'\oplus A$ where $A$ is free abelian with the trivial Lie bracket. As conjugacy separability of an abelian Lie ring is logarithmic, by \cref{prop:directSumDecomposition} we may restrict ourselves to $\mf g'$.

		Let $n$ be some positive integer and let $a,c\in\mf g$, be arbitrary such that $\gnorm{a}\leq n$, $\gnorm{c}\leq n$ and such that $c\notin[a,\mf g]$.We can assume that $a\in Z(\mf g)^\perp$, indeed if $a=a_z+a_{\perp}$ with $a_z \in Z(\mf g), a_{\perp} \in Z(\mf g)^\perp$, then $[a,\mf g]=[a_\perp,\mf g]$.
		Furthermore, we may assume that $c\in Z(\mf g)$. If this is not the case, then by the residual finiteness growth of free abelian groups, we might find a quotient of size at most $C\log(\gnorm{c})$ such that $c+Z(\mf g)$ does not vanish in this quotient.
				
		If $a =0$, then by \cref{prop:pnondevider} we can find some prime number $p$ not dividing the index of $c$ such that $p$ is no larger than $2C\log n$. Then $c$ does not vanish in $\frac{\mf g}{p\mf g}$. So we can assume that $a \neq 0$, and consider now $L = \ind a$. 
		
		First assume that $L$ does not divide $\ind c$.
		Then on the one hand $c$ does not vanish in $\frac{\mf g}{L\mf g}$, but on the other hand $a$ does vanish in this space, and thus also $[a,\mf g]\subset L\mf g$. As $L\leq n$, we have that $\abs{\frac{\mf g}{L\mf g}}\leq n^h\leq n^{h_z\lfloor h-h_z\rfloor_{\mathrm{odd}}}$ and thus the result follows.
		
		In the other case, $L$ divides $\ind c$. Let then $a_0,c_0\in\mf g$ such that $La_0=a$ and $Lc_0=c$. Clearly $a_0$ has index $1$. We have that $c_0\notin[a_0,\mf g]$ as otherwise, $Lc_0\in L[a_0,\mf g]=[La_0,\mf g]$. Notice that the Guivarch norm of $a_0$ is at most $\frac{\gnorm{a}}{L}$.
		By \cref{prop:AlgUpperBoundIndex1} there exists some quotient $\varphi_0:\mf g\rightarrow \mf q_0$ such that $\varphi_0(c_0)\notin \varphi_0([a_0,\mf g])$ and such that $q=\abs{\mf q_0}$ is bounded above by  $$C\max\set{\gnorm{a_0},\log(\max\set{\gnorm{a_0},\gnorm{c_0}})}^{h(Z(\mf g))\lfloor h(Z(\mf g)^\perp)\rfloor_{\mathrm{odd}}}$$ for some constant $C$ only depending on the choice of norms.
		Take $K_0$ the kernel of $\varphi_0$. As $\mf g$ is torsion-free, and $c_0\notin [a_0,\mf g]+K_0$ it follows that $Lc_0\notin [La_0,\mf g]+LK_0$.
		Let $\varphi:\mf g\rightarrow\mf q$ then be the quotient with $LK_0$ as kernel, then $\varphi(c)\notin\varphi([a,\mf g])$.
		
		The quotient $\mf q$ has order exactly the index of $LK_0$ in $\mf g$, which is equal to $L^{h(\mf g)} q$. As $L$ is bounded by $Cn$, the $\log$ appearing in the bound for $q$ will be dominated by any polynomial, so it suffices to bound the first part. As ${h(\mf g)}< {h(Z(\mf g))\lfloor h(Z(\mf g)^\perp)\rfloor_{\mathrm{odd}}}$, we thus have that \begin{align*} \vert \mf q \vert &\leq C\left(\frac{\gnorm{a}}{L}\right)^{h(Z(\mf g))\lfloor h(Z(\mf g)^\perp)\rfloor_{\mathrm{odd}}}L^{h(\mf g)}  \\ &\leq Cn^{h(Z(\mf g))\lfloor h(Z(\mf g)^\perp)\rfloor_{\mathrm{odd}}}.\end{align*}
			This concludes the proof.

	\end{proof}
Note that more precise estimates on $L$ in the proof above would lead in the case $h_2(G)=1$ to the upper bound from \cite{pengitore2015effective}.
	
	\subsection{Sharpness of the $2$-step Case}\label{sec:AlgBoundScharp}
	In \cref{sec:AlgUpperBound} we have constructed an upper bound on the conjugacy separability of $2$-step nilpotent groups or, equivalently, Lie rings. In this section we construct a family of Lie rings $\{\mf g_{k,l}\}$ with $h_2(\mf g_{k,l})= l$, $h(\mf g_{k,l}) = k \geq 2+2l$. We demonstrate that $\conj_{\mf g_{k,l}}\simeq n^{l(\lfloor k-l\rfloor_\mathrm{odd})}$ and thus that these Lie rings match the bound of \cref{sec:AlgUpperBound}. Notice that we might restrict ourselves to the case where $k-l$ is odd, we can then obtain the bound for the even case by taking the direct product with $\ZZ$.
	
	Define the $2$-step nilpotent Lie ring $\mf g_{k,l}$ on the basis
	$$\{
	u_1,u_2,x_i,e_j\mid i\in [0,k-l-3],j\in[1,l]
	\}$$
	where the commutator relations are given by the following graph:
	\[\begin{tikzcd}
		& {u_1} & {u_2} \\
		&&& {} & {} \\
		{x_0} & {x_1} & {x_2} & {x_3} & {x_4} & {}
		\arrow["{e_1}"{description}, from=1-2, to=3-1]
		\arrow["{e_2}"{description}, from=1-2, to=3-2]
		\arrow["{e_3}"{description}, from=1-2, to=3-3]
		\arrow["{e_4}"{description}, from=1-2, to=3-4]
		\arrow["{e_5}"{description}, from=1-2, to=3-5]
		\arrow["{e_2}"{description}, from=1-3, to=3-1]
		\arrow["{e_3}"{description}, from=1-3, to=3-2]
		\arrow["{e_4}"{description}, from=1-3, to=3-3]
		\arrow["{e_5}"{description}, from=1-3, to=3-4]
		\arrow[dotted, no head, from=2-4, to=2-5]
		\arrow["{e_1}"{description}, from=3-2, to=3-3]
		\arrow["{e_1}"{description}, from=3-4, to=3-5]
		\arrow[dotted, no head, from=3-5, to=3-6]
	\end{tikzcd}\]
	To be more precise, we have $[u_i,x_j]=e_{i+j}$ for $i \in \{1,2\}$, $j\geq 0$ and $i+j\leq l$, and $[x_{2i-1},x_{2i}]=e_1$ for $i\in [1,\frac{k-l-3}{2}]$. All other commutators vanish.

		If we want to demonstrate a lower bound to $\cd_{\mf g}(a,c)$, then a priori we need to check if $c$ belongs to the image of $[a,\mf g]$ in all quotients of $\mf g$. The following lemma allows us to impose some constraints on the considered quotients. 
	\begin{lemma}\label{prop:factorsThroughACommut}
		Let $\mf g$ be a $2$-step nilpotent Lie ring and $a,c\in \mf g$ such that $c\notin[a,\mf g]$. Let $I\normal \mf g$ be such that $c\notin[a,\mf g]+I$ and such that $[\mf g:I]=\cd_{\mf g}(a,c)$, then $[a,\mf g]\subset I$.
	\end{lemma}
	\begin{proof}
	As the quotient
$\mf g/(I+[a,\mf g])$ still separates $c$ from $[a,\mf g]$, it must hold that $I+[a,\mf g]=I$.
	\end{proof}

Our goal is to show that for $a_p=pu_1-u_2$ we have $\cd_{\mf g}(a_p,e_l)\geq p^{l(k-l)}$.
	For this we consider the following quotients of $\mf g_{k,l}$:
	
	\begin{definition}
		Let $K_p\subset \mf g_{k,l}$ be the ideal generated by $\set{pe_i-e_{i+1}\mid i\in\range{l-1}}$. Define then $$\mf g_{k,l,p}=\frac{\mf g_{k,l}}{K_p}.$$
	\end{definition}
	To motivate this definition, notice that $K_p \subset [a_p,\mf g_{k,l}]$ as for $i\in[1,l-1]$, we have that $pe_i-e_{i+1}=[a_p,x_{i-1}]$. It is clear that the commutator subgroup of $\mf g_{k,l,p}$ is a cyclic subgroup generated by $e_1$ as the elements $e_i$ and $p^{i-1}e_1$ project to the same element. One checks that $[a_p,\mf g_{k,l}]+K_p= \langle p^l e_1\rangle+K_p.$ Notice that \cref{prop:factorsThroughACommut} allows us to pass from $\cd_{\mf g_{k,l}}$ to $\cd_{\mf g_{k,l,p}}$.
			
	As $\mf g_{k,l,p}$ has cyclic derived Lie subring, the Lie bracket induces a bilinear form$$B:\frac{\mf g_{k,l,p}}{\langle e_1\rangle}\times \frac{\mf g_{k,l,p}}{\langle e_1\rangle}\rightarrow \ZZ$$
	
Let $V \subset \frac{\mf g_{k,l,p}}{\langle e_1\rangle}$ be the submodule of codimension-$1$ generated by $\{u_1,x_0,x_1,\cdots ,x_{k-l-3}\}$.	The bilinear form $B$ induces a linear map $$\begin{matrix}
		B_*&:&V&\rightarrow&V^*\\
		&:&a&\mapsto&[a,\_]\mid_V
	\end{matrix}.$$
	
	\begin{lemma}\label{prop:BInvertible}
		For every $i\geq 1$, the map $\frac{\ZZ}{p^i\ZZ}\otimes B_*$ is an isomorphism from $\frac{\ZZ}{p^i\ZZ}\otimes V$ to $\frac{\ZZ}{p^i\ZZ}\otimes V^*$.
	\end{lemma}
	\begin{proof}
		Let $\beta=\{u_1,x_0,x_1,\cdots ,x_{k-l-3}\}$, a basis of $V$ with dual basis $\beta^*$, let $\delta$ be the map from $V^*$ to $V$ mapping the dual basis vectors to the corresponding basis vectors. Clearly $\delta$ is an isomorphism, it thus suffices to show that $\frac{\ZZ}{p^i\ZZ}\otimes(\delta\comp B_*)$ is an isomorphism. For this it suffices to show that the determinant of $\frac{\ZZ}{p^i\ZZ}\otimes(\delta\comp B_*)$ is invertible in $\frac{\ZZ}{p^i\ZZ}$ or thus that the determinant of $\delta\comp B_*$ is not a multiple of $p$. To show this, it in turn suffices that $\frac{\ZZ}{p\ZZ}\otimes(\delta\comp B_*)$ is invertible.
		
		Notice that for $i\geq 2$, $e_i\in \langle p e_1\rangle$. It thus follows that $\frac{\ZZ}{p\ZZ}\otimes B(u_1,x_j)$ vanishes unless $j=0$. We also have that $\frac{\ZZ}{p\ZZ}\otimes B(x_i,x_j)$ vanishes unless $\{i,j\}=\{2n-1,2n\}$ for some $n\in \NN_{>0}$.
		We thus have that $$\frac{\ZZ}{p\ZZ}\otimes B(\_,\_)\mid_{V\times V}$$ is the standard symplectic form.
        The corresponding map $\frac{\ZZ}{p\ZZ}\otimes(\delta\comp B_*)$ has the same matrix and is thus invertible.
	\end{proof}
	
	\begin{lemma}\label{prop:Vinjective}
		Let $a\in V$. If for all $b\in V$ holds that $B(a,b)\cong 0\mod p^l$ then $a\in p^lV$.
	\end{lemma}
	\begin{proof}
		This follows immediately from \Cref{prop:BInvertible}. 
	\end{proof}
	
	From this we obtain a first lower bound on $\cd$:
	\begin{lemma}\label{prop:cdForP}
		$\cd_{\mf g_{k,l,p}}(a_p,e_l)\geq p^{l(k-l)}$
	\end{lemma}
	\begin{proof}
		Suppose that $\varphi:\mf g_{k,l,p}\rightarrow\mf q$ is a finite quotient of minimal size such that $\varphi(e_l)\notin \varphi([a_p,\mf g_{k,l,p}])$. By \Cref{prop:factorsThroughACommut}, $\varphi$ factors over $\frac{\mf g_{k,l,p}}{\langle p^le_1\rangle}$. Denote $$\overline{\varphi}:\frac{\mf g_{k,l,p}}{\langle p^le_1\rangle}\rightarrow\mf q$$ the induced map.
		Note that in this case $\varphi(\langle e_1\rangle)=\frac{\langle e_1\rangle}{\langle p^le_1\rangle}$, as $\varphi(p^l e_1) = 0$, but any non-trivial quotient would map $e_l=p^{l-1}e_1$ to $0$. 
		
		Let now $V^+=V\oplus (\frac{\ZZ}{p^l\ZZ}\otimes\langle e_1\rangle)\subset\frac{\mf g_{k,l,p}}{\langle p^le_1\rangle}$, and consider the map $\varphi^+: V^+\rightarrow \frac{\ZZ}{p^l\ZZ}\otimes V$ mapping $e_1$ to $0$. Notice that $V^+$ is a Lie subring and that $\varphi^+$ is a Lie ring morphism. Let $a\in V^+\cap\ker(\overline{\varphi})$, we claim that $a\in p^l V+\langle e_1\rangle$. Indeed suppose $a$ is not of this form, then by \Cref{prop:Vinjective}, there exists some $b\in V$ such that $[a,b]\in \frac{\langle e_1\rangle}{\langle p^le_1\rangle}$ is non-trivial. Taking some multiple of $b$, we may assume that $[a,b]=p^{l-1}e_1 + \langle p^l e_1 \rangle$. As $\overline\varphi(p^{l-1}e_1)$ is non-trivial in $\mf q$ it follows that $\overline\varphi(a)$ also is non-trivial, leading to the claim. 
        
        As a consequence, we have that $\ker(\overline\varphi\mid _{V^+})\subset\ker( \varphi^+)$ and thus that the map $\varphi^+$ factors over $\overline{\varphi}\mid_{V^+}$. Notice that $\frac{\ZZ}{p^l\ZZ}\otimes V$ has $p^{l(k-l-1)}$ elements. As the factor map $\tilde \varphi:\overline{\varphi}(V^+)\rightarrow \frac{\ZZ}{p^l\ZZ}\otimes V$ has kernel at least $\frac{\langle e_1\rangle}{\langle p^le_1\rangle}$, it follows that $\mf q$ has at least $p^l p^{l(k-l-1)} = p^{l(k-l)}$ elements.

	\end{proof}
	\begin{corollary}\label{prop:BoundIsSharp}
		If $k-l$ is odd, then
		$n^{l(k-l)} \preceq \conj_{\mf g_{k,l}}(n)$.
	\end{corollary}
	\begin{proof}
		Notice that $a_p$ and $e_l$ are of norm no more then $p+1$ and $1$ respectively. By \Cref{prop:comparegeometric} it thus suffices to show that $\cd_{\mf g_{k,l}}(a_p,e_l)\geq p^{l(k-l)}$ for all primes $p$.
		Let $\varphi:\mf g_{k,l}\rightarrow\mf q$ be a finite quotient such that $\varphi(e_l)\notin\varphi([a_p,\mf g_{k,l}])$, and suppose that $\mf q$ is of minimal size such that these conditions hold, then by \cref{prop:factorsThroughACommut}, this morphism factors through $\mf g_{k,l,p}$. The result follows by \Cref{prop:cdForP}.
	\end{proof}

	\bibliographystyle{plain}
	\bibliography{refs}
\end{document}